\documentclass[12pt]{amsart}
\usepackage{latexsym,amssymb,amsmath}

\numberwithin{equation}{section}
\newtheorem{theorem}{Theorem}[section]
\newtheorem{lemma}[theorem]{Lemma}

\newtheorem{corollary}[theorem]{Corollary}

\newtheorem{question}[theorem]{Question}

\theoremstyle{definition}
\newtheorem{definition}[theorem]{Definition} 
\newtheorem{remark}[theorem]{Remark}

\begin{document}


\newcommand{\del}{\operatorname{del}}
\newcommand{\lk}{\operatorname{lk}}
\newcommand{\reg}{\operatorname{reg}}
\newcommand{\pd}{\operatorname{pd}}

 
\title[SCM bipartite graphs]
{Sequentially Cohen-Macaulay bipartite graphs: vertex decomposability 
and regularity}
\thanks{Version: June 1, 2009}
 
\author{Adam Van Tuyl}
\address{Department of Mathematical Sciences \\
Lakehead University \\
Thunder Bay, ON P7B 5E1, Canada}
\email{avantuyl@lakeheadu.ca}
\urladdr{http://flash.lakeheadu.ca/$\sim$avantuyl/}
 

\begin{abstract}
Let $G$ be a bipartite graph with edge ideal
$I(G)$ whose quotient ring $R/I(G)$ is sequentially Cohen-Macaulay.
We prove: (1) the independence complex
of $G$ must be vertex decomposable, and (2)
the Castelnuovo-Mumford
regularity of $R/I(G)$ can be determined from the invariants of $G$. 
\end{abstract}
 
\maketitle


\section{Introduction} 

Let $G = (V_G,E_G)$ denote a finite simple graph with
vertices $V_G = \{x_1,\ldots,x_n\}$ and edge set $E_G$.  By
identifying the vertices with the variables in the polynomial
ring $R = k[x_1,\ldots,x_n]$, we can associate to each simple
graph $G$ a monomial ideal $I(G) = (\{x_ix_j ~|~ \{x_i,x_j\} \in E_G\})$.
The ideal $I(G)$ is the {\bf edge ideal} of $G$ and was first
introduced by Villarreal \cite{V}.    Also associated to $G$
 is a simplicial complex $\Delta(G)$, called the
{\bf independence complex}, whose faces are the independent
sets of the graph $G$.  That is, $F \in \Delta(G)$ 
if and only if the set $F$
is an independent set of  vertices.  The independence complex
is the simplicial complex associated to $I(G)$ via
the Stanley-Reisner correspondence.

We call a graph $G$ a {\bf sequentially Cohen-Macaulay graph}
if the corresponding ring $R/I(G)$ is sequentially
Cohen-Macaulay (SCM).   In this paper 
we consider SCM graphs $G$ that are also {\bf bipartite},
that is, we can partition $V_G$ as $V_G = V_1 \cup V_2$
so that every edge $e \in E_G$ has one endpoint
in $V_1$ and the other endpoint in $V_2$.  SCM bipartite graphs,
which includes the set of Cohen-Macaulay bipartite graphs,
have been studied in \cite{CFF,EV,F,FH,FVT,HH2,HHZ,VTV}.
With the additional assumption that $G$ is bipartite,
these papers have shown that the 
algebraic property of being SCM (or CM) is really
a combinatorial property.  

In this short note we present two new results that further
highlight the fact that being bipartite and SCM is really
a combinatorial property.  Our first main theorem (Theorem
\ref{maintheorem1}) shows that $G$ is SCM if and only if 
$\Delta(G)$ is a vertex decomposable simplicial complex.
 The notion of a vertex decomposable simplicial complex
was independently introduced to the study of edge ideals
by Dochtermann and Engstr\"om \cite{DE} and Woodroofe \cite{W}.
Our result gives a new proof that $\Delta(G)$ must
also be shellable (as first proved by the author and Villarreal \cite{VTV}).
Our second result (Theorem \ref{maintheorem2}) is a formula for the
Castelnuovo-Mumford regularity in terms of the number
of 3-disjoint edges in a graph.
We recover Zheng's \cite{Z} formula for the regularity
of the edge ideals of trees as corollary.

Although not discussed directly in this paper, unmixed
bipartite graphs (graphs whose minimal vertex covers all
have the same cardinality) were studied
in \cite{HHO,K,MM,V3}.  In this situation, the algebraic properties
of the ring $R/I(G)$ are again highly connected with
the invariants of the graph.  Note that the intersection
of the set of unmixed bipartite graphs and the set 
of SCM graphs is precisely the set of CM graphs.  However, little
appears to be known about the edge ideals of bipartite
graphs that are neither unmixed or SCM;  this appears to be an area
that requires further exploration.  In fact, at the end of
the paper, we raise a question about the regularity of these ideals.

{\bf Acknowledgments.} Part of this paper was written
when the author was visiting the Universit\`a di Catania,
the Universit\`a di Genova, and the Politecnico di Torino in March 2009.  
The author would like
to thank his hosts Elena Guardo (Catania), Tony Geramita (Genova),
and Enrico Carlini (Torino)
for their hospitality.  The author
received financial support from GNSAGA and NSERC while working
on this project.


\section{SCM and Vertex Decomposable Graphs}

In this section we show that
sequentially Cohen-Macaulay bipartite graphs have
independence complexes that are vertex decomposable.
A {\bf simplicial complex} $\Delta$ on
$V = \{x_1,\ldots,x_n\}$ is a collection of subsets of $V$ 
such that:
(1) $\{x_i\} \in \Delta$ for $i =1,\ldots,n$, and (2) if $F \in \Delta$
and $G \subseteq F$, then $G \in \Delta$.  Elements of 
$\Delta$ are called the {\bf faces} of $\Delta$, and the
maximal elements, with respect to inclusion, are called the {\bf facets}.
A simplicial complex is {\bf pure} if all its facets have the
same cardinality.

Vertex decomposability was first introduced by Provan and Billera \cite{PB}
in the pure case, and extended to the non-pure case
by Bj\"orner and Wachs \cite{BW,BW1}.  It is defined in terms
of the deletion and link;   
if $F \in \Delta$ is a face, then the {\bf link}
of $F$ is the simplicial complex
\[\lk_{\Delta}(F) = \{ H \in \Delta ~|~ H \cap F = \emptyset ~\mbox{and}~~
G \cup F \in \Delta\},\]
while the {\bf deletion} of $F$ is the simplicial complex
\[\del_{\Delta}(F) = \{H \in \Delta ~|~ H \cap F = \emptyset \}.\]
When $F = \{x\}$ is a single vertex, we abuse
notation and write $\del_\Delta(x)$ and $\lk_\Delta(x)$.

\begin{definition}
Let $\Delta$ be a simplicial complex on the vertex
set $V = \{x_1,\ldots,x_n\}$.  Then $\Delta$
is {\bf vertex decomposable} if either:
\begin{enumerate}
\item[$(i)$] The only facet of $\Delta$ is $\{x_1,\ldots,x_n\}$,
i.e., $\Delta$ is a simplex, or $\Delta = \emptyset$.
\item[$(ii)$]  there exists an $x \in V$ such that
$\del_{\Delta}(x)$ and $\lk_{\Delta}(x)$ are vertex decomposable,
and such that every facet of $\del_{\Delta}(x)$ is a facet of
$\Delta$.  
\end{enumerate}
 If $\Delta$ is pure, we call $\Delta$ {\bf pure vertex
decomposable}.
\end{definition}

Let $G = (V_G,E_G)$ be a graph;  an {\bf independent
set} of $G$ is a subset $F \subseteq V_G$ such
that $e \not\subseteq F$ for every $e \in E_G$.
The {\bf independence complex}
of $G$, denoted $\Delta(G)$, is the simplicial complex on $V_G$ 
with face set
$\Delta(G) = \{ F \subseteq V_G ~|~ \mbox{$F$ is an independent
set of $G$}\}.$
Since we want to know when $\Delta(G)$
is vertex decomposable, we introduce
the terminology:

\begin{definition}
A finite simple graph $G$ is a {\bf vertex decomposable graph} (or
simply, vertex decomposable) if the independence complex $\Delta(G)$
is vertex decomposable.
\end{definition}

To determine if a graph is vertex
decomposable, we can always make the assumption that
the graph is connected.  The next lemma is Lemma 20 in \cite{W}:

\begin{lemma}\label{connectedcomponent}
Let $G_1$ and $G_2$ be two graphs such that $V_{G_1} \cap
V_{G_2} = \emptyset$, and set $G = G_1 \cup G_2$.
Then $G$ is vertex decomposable if and only if $G_1$ and
$G_2$ are vertex decomposable.
\end{lemma}

For any subset $S \subseteq V_G$ in $G$, we let $G \setminus S$ denote
the graph obtained by removing all the vertices of $S$ from
$G$, and any edge which has at least one of its endpoints in $S$.
When $S = \{x\}$ we abuse notation and write $G \setminus x$.
Also, we let $N(x) := \{y \in V_G ~|~ \{x,y\}
\in E_G\}$ be the set of {\bf neighbours} of $x$.
To prove our main result, we will proceed by
induction;  the following lemma 
(see \cite[Lemma 4.2]{DE}) will facilitate this induction:

\begin{lemma}\label{simplicialvertex}
Let $G$ be a graph, and suppose that $x,y \in V_G$ are 
two vertices such that $\{x\} \cup N(x) \subseteq \{y\} \cup N(y)$.
If $G \setminus y$ and $G \setminus (\{y\} \cup N(y))$ are both
vertex decomposable, then $G$ is vertex decomposable.
\end{lemma}

Vertex decomposability was introduced, in part, as a tool
to study the shellability of a simplicial complex. 
We review the relevant connections.

\begin{definition}\label{shellabledefn}
A simplicial complex $\Delta$ is {\bf shellable}  if 
the facets of $\Delta$ can be ordered, say  $F_1,\ldots,F_s$, 
such that
 for all $1\leq i<j\leq s$, there 
exists some $x\in F_j\setminus F_i$ and some 
$\ell\in \{1,\ldots,j-1\}$ with $F_j\setminus F_\ell= \{x\}$. 
If $\Delta$ is pure, we call $\Delta$ {\bf pure shellable}.
\end{definition}

As in \cite{VTV}, we call a graph $G$ a {\bf shellable graph} if
$\Delta(G)$ is a shellable simplicial complex.
A vertex decomposable graph is then shellable because
of the following more general result (see \cite[Theorem 11.3]{BW1}):

\begin{theorem}\label{vd->shellable}
 If $\Delta$ is a vertex decomposable 
simplicial complex, then $\Delta$ is also shellable.
\end{theorem}

A graded $R$-module $M$ is called 
{\bf sequentially Cohen-Macaulay} (over $k$)
if there exists a finite filtration of graded $R$-modules
$0 = M_0 \subset M_1 \subset \cdots \subset M_r = M$
such that each quotient $M_i/M_{i-1}$ is Cohen-Macaulay, and the
 Krull dimensions of the
quotients are increasing:
$\dim (M_1/M_0) < \dim (M_2/M_1) < \cdots < \dim (M_r/M_{r-1}).$
As first shown by Stanley \cite{S}, shellability implies
sequentially Cohen-Macaulayness. 

\begin{theorem}\label{shellable->scm}
Let $\Delta$ be a simplicial complex, and suppose that $R/I_{\Delta}$
is the associated Stanley-Reisner ring.  If $\Delta$ is shellable, then
$R/I_{\Delta}$ is sequentially Cohen-Macaulay.  
\end{theorem}

Before coming to our main result, we need two facts from 
\cite{VTV} about SCM graphs.

\begin{theorem}[{\cite[Theorem 3.3]{VTV}}]\label{removevertex}
Let $x \in V_G$ be any vertex of $G$ and set $G' = G \setminus (\{x\} \cup N(x))$.
If $G$ is SCM, then $G'$ is SCM.
\end{theorem}

\begin{lemma}[{\cite[Theorem 3.7]{VTV}}]\label{onevertexdegree1}
Let $G$ be a bipartite graph.  If $G$ is SCM, then there exists
a vertex $x \in V_G$ such that $\deg x = 1$.
\end{lemma}

We can now prove the main theorem of this section:

\begin{theorem}\label{maintheorem1}
Let $G$ be a bipartite graph.
Then the following are equivalent:
\begin{enumerate}
\item[$(i)$] $G$ is SCM.
\item[$(ii)$] $G$ is shellable.
\item[$(iii)$] $G$ is vertex decomposable.
\end{enumerate}
\end{theorem}

\begin{proof}
Note that $(iii) \Rightarrow (ii) \Rightarrow (i)$ always holds for
any graph $G$ by Theorems \ref{vd->shellable} and \ref{shellable->scm}.  It
suffices to show that when $G$ is bipartite, then $(i) 
\Rightarrow (iii)$.

We do a proof by induction on $n$, the number of vertices.  When
$n=2$, then $G$ consists of a single edge.  This graph is SCM (in fact, CM),
and the independence complex is a simplex, hence vertex decomposable.
We therefore suppose that $n > 2$.  
By Lemma \ref{onevertexdegree1}, we know that there is 
a vertex $x$ of degree 1;  let $y$ denote
the unique neighbour of $x$.

Set $G_1 = G \setminus (\{x\} \cup N(x))$ and $G_2 = G \setminus (\{y\} \cup N(y))$.
By Theorem \ref{removevertex}, both of these graphs are SCM, and thus by
induction, $G_1$ and $G_2$ are vertex decomposable.  
Let $G_1'$ be the graph obtained by adding the isolated vertex $x$ to $G_1$.  
Because $x$ is only adjacent to $y$, the graph $G_1'$ is the
same as the graph $G \setminus y$.  Furthermore,
since $G_1$ is vertex decomposable, then 
so is $G_1'$ by Lemma \ref{connectedcomponent}.
So $G \setminus y$ and $G \setminus (\{y\} \cup N(y))$ are vertex decomposable,
and because $\{x\} \cup N(x) \subseteq \{y\} \cup N(y)$, 
Lemma \ref{simplicialvertex} thus implies that $G$ is vertex decomposable.
\end{proof}

\begin{remark}
The equivalence of $(i)$ and $(ii)$ was first 
proved in \cite{VTV}.  Theorem \ref{maintheorem1} 
further highlights the combinatorial
nature of SCM bipartite graphs.
\end{remark}

The equivalence of $(i)$ and $(ii)$ in the corollary below  was
first proved in \cite{EV}.  In the proof 
(and in the next section) we require the following notion:
a subset $W \subseteq V_G$ is a {\bf vertex cover} if for every
$e \in E_G$, we have $W \cap e \neq \emptyset$.  A {\bf minimal
vertex cover} is any vertex cover $W$ with the property
that for every $x \in W$, $W \setminus \{x\}$ is not a vertex cover.
 
\begin{corollary} Let $G$ be a bipartite graph.  Then the following
are equivalent:
\begin{enumerate}
\item[$(i)$] $G$ is Cohen-Macaulay.
\item[$(ii)$] $G$ is pure shellable.
\item[$(iii)$] $G$ is pure vertex decomposable.
\end{enumerate}
\end{corollary}

\begin{proof}
A graph $G$ is Cohen-Macaulay if and only if $G$ is SCM 
and $I(G)$ is unmixed, i.e., all of its associated
primes have the same height (see \cite[Lemma 3.6]{FVT}).  
But the associated primes
of an edge ideal correspond to the minimal vertex covers
of $G$ (see \cite{V2}).  The complement of a vertex
cover is an independent set, i.e., a face
of $\Delta(G)$.  Since all the minimal vertex covers
have the same cardinality, so do the facets of $\Delta(G)$,
i.e., $\Delta(G)$ is pure.   So, $(i)$ implies $(iii)$
since $G$ is SCM, and thus $G$ is vertex decomposable by
Theorem \ref{maintheorem1}
and $\Delta(G)$ is pure.  The implications
$(iii) \Rightarrow (ii) \Rightarrow (i)$ hold for any graph.
\end{proof}


\section{The regularity of SCM bipartite graphs}

In this section we will give a formula for the Castelnuovo-Mumford
regularity of $R/I(G)$ when $G$ is SCM and bipartite in terms of the graph $G$.

Associated to an $R$-module $M$ is a {\bf
minimal free graded resolution} of the form:
\[
0 \rightarrow \bigoplus_j R(-j)^{\beta_{p,j}(M)}
\rightarrow \bigoplus_j R(-j)^{\beta_{p-1,j}(M)}
\rightarrow \cdots
\rightarrow  \bigoplus_j R(-j)^{\beta_{0,j}(M)}
\rightarrow M \rightarrow 0
\]
where $p \leq n$ and $R(-j)$ is the $R$-module obtained by shifting
the degrees of $R$ by $j$.  The number $\beta_{i,j}(M)$, the $ij$th 
{\bf graded Betti number} of $M$, equals the number of generators
of degree $j$ in the $i$th syzygy module.  The {\bf Castelnuovo-Mumford
regularity} (or simply regularity) of $M$ is 
\[\reg(M):= \max\{j-i ~|~ \beta_{i,j}(M) \neq 0\}.\]
The {\bf projective dimension} of $M$ is the length of the minimal
free resolution, that is,
\[\pd(M) := \max\{i ~|~ \beta_{i,j}(M) \neq 0 ~~\mbox{for some $j$}\}.\]

As in \cite{HVT}, we say two edges $\{x,y\}$ and $\{w,z\}$ of
$G$ are {\bf 3-disjoint} if the induced subgraph of $G$ on
$\{x,y,w,z\}$ consists of exactly two disjoint edges.  This
condition is equivalent to saying that in the complement graph $G^c$,
i.e., the graph whose edges are precisely the non-edges of $G$,
the induced graph on $\{x,y,w,z\}$ is an induced four-cycle.  Zheng
\cite{Z} (and also in \cite{K}) 
called edges of this type {\bf disconnected}.
For any graph $G$,
we let $a(G)$ denote the maximum number of pairwise 
3-disjoint edges in $G$.

For any graph $G$, Katzman provided the following lower
bound on $\reg(R/I(G))$.

\begin{lemma}[{\cite[Lemma 2.2]{Ka}}]  \label{lowerbound}
For any graph $G$,
$\reg(R/I(G)) \geq a(G)$.
\end{lemma}
\noindent
There are examples
of graphs which have $\reg(R/I(G)) > a(G)$.  However,
as shown below, we have an equality if $G$ is SCM and bipartite.

Given a graph $G$,
we can associate to $G$ another monomial ideal called the 
{\bf cover ideal}:
\[I(G)^{\vee} = \langle \{ x_{i_1}\cdots x_{i_r} ~|~ 
W = \{x_{i_1},\ldots,x_{i_r}\} ~~\mbox{is a minimal vertex cover of $G$}\}\rangle.\]
The notation $I(G)^{\vee}$ is used because $I(G)^{\vee}$ is the Alexander
Dual of the edge ideal.
We require a 
result of Terai about the Alexander Dual of a square-free monomial
ideal.

\begin{theorem}[\cite{T}]\label{Terai} 
Let $I$ be an square-free monomial ideal.
Then $\pd(I^{\vee}) = \reg(R/I)$.
\end{theorem}

We now prove the main result of this section.

\begin{theorem}\label{maintheorem2}
If $G$ is a SCM bipartite graph
with $a(G)$ pairwise $3$-disjoint edges, then 
\[\operatorname{reg}(R/(I(G)) = a(G).\]
\end{theorem}

\begin{proof}
By Lemma \ref{lowerbound}, it suffices to show that
$a(G)$ is an upper bound.  By Theorem \ref{Terai}, we 
have  $\reg(R/I(G)) = \pd(I(G)^{\vee})$, so it is enough to
prove that $\pd(I(G)^{\vee}) \leq a(G)$.
We proceed by induction on $n$.
If $n = 2$, then $G$ is single edge $\{x,y\}$, 
and $I(G)^{\vee} = (x,y)$ which has $\pd(I(G)^{\vee}) = 1 \leq 1= a(G)$.

Suppose that $n >2$.  By Lemma \ref{onevertexdegree1}, there exists
a vertex $x$ of degree one.  Let $y$ be the unique neighbour
of $x$, and let $N(y) = \{x,x_{i_2},\ldots,x_{i_t}\}$ be
the neighbours of $y$.
Observe that if $W$ is any
minimal vertex cover of $G$, then it cannot contain
both $x$ and $y$, because if it did, then $W \setminus \{x\}$
would still be a vertex cover.  Also, if $y \not\in W$, then
$N(y) \subseteq V$.  
Set $G' = G \setminus (\{y\} \cup N(y))$ and $G'' = G \setminus (\{x\}
\cup N(x))$.  Let $I(G')^{\vee}$, respectively, $I(G'')^{\vee}$,
denote the cover ideal of $G'$, respectively $G''$, but
viewed as ideals of $R = k[x_1,\ldots,x_n]$.
We need two facts:
\vspace{.2cm}

\noindent
{\bf Claim 1.} $~~~I(G)^{\vee} = xx_{i_2}\cdots x_{i_t} I(G')^{\vee} + yI(G'')^{\vee}.$

\noindent
{\bf Proof of Claim.}  Let $m \in I(G)^{\vee}$ be a generator.  Then
$m$ corresponds to a minimal 
vertex cover of $G$.  It contains
either $x$ or $y$, but not both.  If it contains $x$, then it
contains $N(y) = \{x_{i_2},\ldots,x_{i_t}\}$.  So
$m = xx_{i_2}\cdots x_{i_t}m'$ where $m'$ must be a cover
of $G'$.  So $m \in  xx_{i_2}\cdots x_{i_t} I(G')^{\vee}$.  
If $x \nmid m$, then $m = ym'$, and $m'$ must be
a cover of $G''$.  So $m \in  yI(G'')^{\vee}$, thus showing one
containment.  The reverse direction is proved similarly; each generator
on the right hand side must be a cover of $G$, and so belong to $I(G)^{\vee}$.
\hfill$\Box$
\vspace{.2cm}

\noindent
{\bf Claim 2.} $~~xx_{i_2}\cdots x_{i_t} I(G')^{\vee} \cap yI(G'')^{\vee} = 
yxx_{i_2}\cdots x_{i_t}I(G')^{\vee}.$

\noindent
{\bf Proof of Claim.}  
We have $yxx_{i_2}\cdots x_{i_t}I(G')^{\vee} \subseteq  xx_{i_2}\cdots x_{i_t} I(G')^{\vee}$.  If $m \in yxx_{i_2}\cdots x_{i_t}I(G')^{\vee}$, then $ m = yxm'$ where 
$m'$ is a cover of $G''$ so it is also in the second ideal, and
thus in the intersection.  For the other containment,
if $m$ is the intersection, then $m = xx_{i_2}\cdots x_{i_t}m'$
with $m' \in I(G')^{\vee}$.  But since $m \in yI(G'')^{\vee}$,
we have that $y|m$, and this implies that $m'= ym''$.  
So $m = yxx_{i_2}\cdots x_{i_t}m''$.  But $m''$ has
to be in $I(G')^{\vee}$ because none of its generators 
are divisible by $y$.  
\hfill$\Box$

As a consequence of Claims 1 and 2, we 
have a short exact sequence
\[0 \rightarrow yxx_{i_2}\cdots x_{i_t}I(G')^{\vee}
\longrightarrow  xx_{i_2}\cdots x_{i_t} I(G')^{\vee} \oplus yI(G'')^{\vee}
\longrightarrow I(G)^{\vee} 
\longrightarrow 0.\]
In particular, the short exact sequences gives the
following bound on $\pd(I(G))^{\vee}$:
\[\pd(I(G))^{\vee} \leq \max\{\pd(yxx_{i_2}\cdots x_{i_t}I(G')^{\vee})+1,
\pd(xx_{i_2}\cdots x_{i_t} I(G')^{\vee}), \pd(yI(G'')^{\vee})\}.\]
However, for any monomial ideal $I$ and monomial $m$ with
the property that the support of $m$ is disjoint from
the support of any generator of $I$, we have $\pd(mI) = \pd(I)$.
Thus
\[\pd(I(G)^{\vee}) \leq \max \{\pd(I(G')^{\vee})+1,\pd(I(G'')^{\vee})\}.\]
By induction, $\pd(I(G')^{\vee}) +1 \leq a(G') + 1$ and 
$\pd(I(G'')^{\vee}) \leq a(G'')$. Now $a(G'') \leq a(G)$ is clear,
and $a(G') + 1 \leq a(G)$ since we can take the $a(G')$ $3$-disjoint
edges of $G'$ along with the edge $\{x,y\}$ to form a set
of $a(G')+1$ $3$-disjoint edges in $G$.
\end{proof}

As corollaries, we can recover a result of Zheng \cite{Z}, who
first linked the regularity to the invariant $a(G)$,
and a result of the author with Francisco and H\`a \cite{FHVT}.
The corollary is true because 
trees and  CM bipartite graphs belong to the set of SCM bipartite graphs.
Trees, which are always bipartite, were first shown to be SCM
in \cite{F};  alternative proofs were given in
\cite{DE,FVT,VTV,W}.

\begin{corollary}
If $G$ is either a tree or a CM bipartite
graph, then $\reg(R/I(G)) = a(G)$.
\end{corollary}

Kummini \cite{K} recently showed that if $G$
is an {\bf unmixed} bipartite graph, that is, all of its minimal
vertex covers have the same cardinality, then
$\reg(R/I(G)) = a(G)$ also holds.  Thus,
for bipartite graphs, we would like to know an
answer to the question:

\begin{question}
Suppose that $G$ is a bipartite graph that is not
unmixed and not SCM.  What is $\reg(R/I(G))$?
\end{question}

The answer will {\it not} be $a(G)$ in general.  As noted in
\cite{K}, the cycle with eight vertices is
a mixed bipartite graph with two $3$-disjoint edges, but
$\reg(R/I(G)) = 3$.  This graph is also not SCM by
Lemma \ref{onevertexdegree1}.    On the other hand,
H\`a and the author \cite{HVT} proved that $\reg(R/I(G)) \leq  \alpha'(G)$
where $\alpha'(G)$ is the {\bf matching number}, the largest set
of pairwise disjoint edges.  For the same eight cycle, we
have $\alpha'(G) = 4$, and thus $\reg(R/I(G)) < \alpha'(G)$.
It would be nice to determine a formula
for $\reg(R/I(G))$ for all bipartite graphs.


\bibliographystyle{plain}

\end{document}